\documentclass[runningheads]{llncs}

\usepackage{graphicx}
\usepackage[english]{babel}
\usepackage{  amssymb,amsmath,bbm,enumitem,stmaryrd,
		      tikz-cd,mathrsfs,extarrows
			}
\usepackage[colorlinks=false]{hyperref}
\usetikzlibrary{quotes,babel,angles,matrix,calc,patterns,graphs,3d,positioning,arrows.meta}
\DeclareMathOperator{\CM}{\mathbf{CM}}

\DeclareMathOperator{\dom}{dom}

\DeclareMathOperator{\Hom}{Hom}

\DeclareMathOperator{\id}{id}

\DeclareMathOperator{\pr}{\mathsf{pr}}
\DeclareMathOperator{\prbold}{\mathbf{pr}}

\DeclareMathOperator{\prt}{prt}

\DeclareMathOperator{\simto}{\rightarrowtriangle}

\DeclareMathOperator{\total}{tot}

\newcommand{\bbone}{\text{\usefont{U}{bbold}{m}{n}1}}
\MakeRobust{\bbone}
\newcommand{\bbtwo}{\text{\usefont{U}{bbold}{m}{n}2}}
\MakeRobust{\bbtwo}
\newcommand{\bbthree}{\text{\usefont{U}{bbold}{m}{n}3}}
\MakeRobust{\bbthree}
\newcommand{\bbfour}{\text{\usefont{U}{bbold}{m}{n}4}}
\MakeRobust{\bbfour}
\newcommand{\bbfive}{\text{\usefont{U}{bbold}{m}{n}5}}
\MakeRobust{\bbfive}
\newcommand{\bbsix}{\text{\usefont{U}{bbold}{m}{n}6}}
\MakeRobust{\bbsix}
\newcommand{\bbseven}{\text{\usefont{U}{bbold}{m}{n}7}}
\MakeRobust{\bbseven}
\newcommand{\bbeight}{\text{\usefont{U}{bbold}{m}{n}8}}
\MakeRobust{\bbeight}
\newcommand{\bbnine}{\text{\usefont{U}{bbold}{m}{n}9}}
\MakeRobust{\bbnine}
\newcommand{\bbnull}{\text{\usefont{U}{bbold}{m}{n}0}}
\MakeRobust{\bbnull}

\DeclareMathOperator{\Assemblies}{\mathcal{A}\mathit{sm}}

\DeclareMathOperator{\CompMod}{\mathsf{CompMod}}

\DeclareMathOperator{\Sets}{\mathsf{Sets}}
\DeclareMathOperator{\SetsB}{\mathbf{Sets}}

\newcommand{\Grothendieck}[2]{\sum_{#1} #2}

\tikzcdset{simto/.tip={Glyph[glyph math command=rightarrowtriangle]}}

\newcommand{\B}[1]{\mathbf{#1}}
\newcommand{\C}[1]{\mathcal{#1}}
\newcommand{\D}[1]{\mathbb{#1}}
\newcommand{\tot}{\mathrm{tot}}
\renewcommand{\fam}{\mathrm{fam}}

\title{The Grothendieck computability model}
\author{Luis Gambarte\inst{1} \and Iosif Petrakis\inst{2}}
\institute{Ludwig-Maximilians-Universit\"at M\"unchen \email{gambarte@math.lmu.de} \and Universit\`a di Verona \email{iosif.petrakis@univr.it}}

\begin{document}

\maketitle

\begin{abstract}
Translating notions and results from category theory to the theory of computability models of Longley and Normann, we introduce the Grothendieck computability model and the 
first-projection-simula\-{}tion. We prove some basic properties of the Grothendieck computability model, and we show that the category of computability models is a type-category, in the sense of Pitts. We introduce the notion of a fibration and opfibration-simulation, and we show that the first-projection-simulation is a split opfibration-simulation.
\end{abstract}

\section{Introduction}

The important role of category theory in computability theory has been emphasised by Cockett and Hofstra in~\cite{Co14,CH08,CH10}, who influenced the work of Longley on computability models and simulations between them in~\cite{Lo95,Lo07,Lo14}. The categorical notion of equivalence between computability models that is studied by Longley and Normann in~\cite{LN15} allowed a better way to ``identify'' seemingly different computability structures. By associating to a computability model $\B C$ its category of assemblies $\Assemblies(\B C)$, Longley and Normann established an equivalence of Morita-type between them. We can summarise the work of Longley and Normann by the phrase ``from computability models to categories''. 

In the previous work~\cite{Pe21,Pe22a,Pe22b} of the second author the converse direction i.e., ``from categories to computability models'', is followed. Given a category $\C C$ and a presheaf $S$ on $\C C$, the total computability model $\CM^{\tot}(\C C;S)$ was introduced, and if $\C C$ is a category with pullbacks and $S$ preserves pullbacks, the partial computability model $\CM^{\prt}(\C C;S)$ was studied. In our joint work in progress~\cite{GP24} the notion of a computability model over a category $\C C$ with a base of computability, a notion close to Rosolini's concept of dominion in~\cite{Ro86}, and a pullback-preserving presheaf on $\C C$ is elaborated, generalising in this way both constructions, that of $\CM^{\tot}(\C C;S)$ and of $\CM^{\prt}(\C C;S)$. Strict computability models are very close to categories of sets and partial functions, but avoiding the equality rules for composition of partial functions (as it is mentioned by Cockett in~\cite{Co14}, p.~16, ``program equality itself is not well-understood''), they possess a more expressive power than categories. Consequently, simulations, the arrows between computability models, avoid equality too, involving certain forcing and tracking relations instead.

Working within the direction ``from categories to computability models'' in this paper too, we ``translate'' the categorical Grothendieck construction and the categorical notion of split (op)fibration to the partial and without equality, or relational framework of computability models. The Grothendieck computability models become then the Sigma-objects, in the sense of Pitts~\cite{Pi00}, in the category of computability models.
We structure this paper as follows:
\begin{itemize}

\item In section~\ref{sec: basic} we include all basic definitions within the theory of computability models necessary to the rest of this paper. Crucial to the definition of the Grothendieck model is our introduction of the computability model $\SetsB$, the computability model-counterpart to the category of sets and functions (Definition~\ref{def: set}). The 
introduced representable-simulations correspond to the representable presheaves (Example~\ref{ex: presheafsim}).

\item In section~\ref{sec: groth} we define the Grothendieck computability model and the corresponding first-projection-simulation (Proposition~\ref{prp: groth1}). We prove some basic properties of the Grothendieck computability model, and we show that the category of computability models $\CompMod$ is a type-category, in the sense of Pitts~\cite{Pi00} (Theorem~\ref{thm: typecat}).

 \item In section~\ref{sec: fibr} we introduce the notion of a (split) fibration and opfibration-simulation and we show that the first-projection-simulation $\B {\pr_1} \colon \Grothendieck{\mathbf{C}}{\pmb{\gamma}} \simto \B C$ is a (split) opfibration-simulation (Proposition~\ref{prp: projopfibr} and Corollary~\ref{cor: split}).

\item In section~\ref{sec: concl} we include some questions and topics for future work.    
\end{itemize}
For all notions and results from category theory that are used here without explanation or proof we refer to~\cite{Ri16}. For various examples of computability models and simulations from higher-order computability theory we refer to \cite{LN15}.

\section{Basic definitions}
\label{sec: basic}

%We begin by recalling the basic definitions of computability models and simulations as due to Longley \cite{Lo14}. 
%We will not use the original definitions using relations but instead the simulation later used by Longley in \cite{LN15} and also by Petrakis in \cite{Pe22b}. 
%The definition is as follows: 

\begin{definition}
A $($strict$)$ computability model $\mathbf{C}$ consists of the following data: 
%\begin{itemize}
 %   \item 
    a class $T$, whose members are called type names; 
    %\item 
    for each $t \in T$ a set $\mathbf{C}(t)$ of data types;
   % \item 
    for each $s,t \in T$ a class $\mathbf{C}[s,t]$ of computable functions, i.e., partial functions from $\mathbf{C}(s)$ to $\mathbf{C}(t)$.
%\end{itemize}
Moreover, for every $r, s, t\in T$
%This data is subject to 
the following hold: 
\begin{enumerate}
    \item The identity $\mathbf{1}_{\mathbf{C}(t)}$ is in $\mathbf{C}[t,t]$.
    \item For every $f \in \mathbf{C}[r,s]$ and $g \in \mathbf{C}[s,t]$ we have that $g \circ f \in \mathbf{C}[r,t]$.
\end{enumerate}
\end{definition}

Next, we describe the computability model of sets and partial functions $\SetsB$, as the computability model-analogue to the category of sets and functions $\Sets$.

\begin{definition}\label{def: set}
The computability model $\SetsB$ has as type names the class of sets and as data types the set $U$ itself, for every type name $U$. If $U, V$ are sets, the computable functions from $U$ to $V$ is the class of partial functions from $U$ to $V$.
%$\Sets[U,V]$ be the set of all total set functions between we obtain a computability model which shall be our analogue of $\Sets$.
\end{definition}

%Longley and Norman give many examples of computability models in \cite{LN15}, which we will not repeat here. 
 A partial arrow $(i,f) \colon a \rightharpoonup b$ 
 %in $\C C$ 
 in a category $\C C$ consists of a monomorphism $i \colon \dom(i) \to a$ and an arrow $f \colon \dom(i) \to b$ in $\C C$. Given a (covariant) presheaf $S \colon \C{C} \to \Sets$, we write $S(i,f)$ instead of $\big( S(i),S(f)\big)$. In~\cite{Pe22b} computability models over categories and presheaves on them were defined in a canonical way.
    %Let $S \colon \C{C} \to \Sets$ be a presheaf.
%and  Instead we focus on the method that allows us to extract a computability model from a category.

%\begin{note}
%    Let $\C{C}$ be an arbitrary category.
%    A partial arrow $(i,f) \colon a \to b$ consists of a monomorphism $i \colon \dom(i) \to a$ and an arbitrary arrow $f \colon \dom(i) \to b$. 
%    Given a presheaf $S \colon \C{C} \to \Sets$ we write $S(i,f)$ instead of $\big( S(i),S(f)\big)$. 
%\end{note}

\begin{definition}\label{def: canonical}
 Let $\C C$ be a category and $S \colon \C C \to \Sets$ a presheaf on $\C C$.
    The total canonical computability model $\CM^{\total}(\C{C};S)$ over $\C C$ and $S$ has as type names the class of objects $\C C_0$ of $\C{C}$ and data types the sets $S(c)$, for every $c \in \C{C}_0$. If $c_1, c_2 \in \C C_0$, the (total) functions from $S(c_1)$ to $S(c_2)$ is the class
    $\{ S(f) ~\vert~ f \in \Hom(c_1, c_2)\}$.
    %functions we define $\CM^{\total}(\C{C};S)(c) = S(c)$.
  %      \item For $c_1,c_2 \in \C{C}_0$ we define $\CM^{\total}(\C{C};S)[c_1,c_2] = \{ S(f) ~\vert~ f \colon c_1 \to c_2\}$.
  %  \end{itemize}
The partial canonical computability model $\CM^{\prt}(\C{C};S)$ over $\C C$ and a pullback-preserving presheaf $S$ has the same type names and data types, while the partial functions from $S(c_1)$ to $S(c_2)$ is the class
    $\{ S(i,f) ~\vert~ (i,f) \colon c_1 \rightharpoonup c_2\}$. 
    %to be the following:
  %  \begin{itemize}
 %       \item The underlying class is $\C{C}_0$ the class of objects of $\C{C}$. 
  %      \item For each $c \in \C{C}_0$ we define $\CM^{\total}(\C{C};S)(c) = %S(c)$.
   %     \item For $c_1,c_2 \in \C{C}_0$ we define $\CM^{\total}(\C{C};S)$$[c_1,c_2] = \{ S(i,f) ~\vert~ (i,f) \colon c_1 \rightharpoonup c_2\}$.
  %  \end{itemize}
\end{definition}

%Note that t
The pullback-preserving property on $S$ is necessary to prove that $\CM^{\prt}(\C{C};S)$ is a computability model.
We can also use the category $\Sets^{\prt}$ of sets and partial functions, and the computability model  
$\CM^{\total}(\Sets^{\prt}, \id_{\Sets^{\prt}})$
    is the computability model $\SetsB$ of Definition~\ref{def: set}. 
    %Note that we could have also obtained this computability model by considering
    %\[ \CM^{\prt}(\Sets,\id_{\Sets}). \]
Next, we describe the arrows in the category of computability models $\CompMod$. A notion of contravariant simulation can also be defined, allowing the  contravariant version of the Grothendieck construction for computability models.

\begin{definition}
\label{def: simulation}
    A simulation $\pmb{\gamma}$ from $\mathbf{C}$ $($over $T)$ to $\mathbf{D}$ $($over $U)$ consists of a class-function $\gamma \colon T \to U$ and a relation $\Vdash_t^\gamma \subseteq \B D\big(\gamma(t)\big) \times \B C(t)$ $($a so-called forcing relation$)$,  for each $t \in T$,
   subject to the following conditions:
    \begin{enumerate}
        \item For each $x \in \mathbf{C}(t)$ there exists some $y \in \B D(\gamma(t))$, such that $y \Vdash_t^\gamma x$.
        \item For each $f \in \mathbf{C}[s,t]$ there exists some $f' \in \mathbf{D}\big[\gamma(s),\gamma(t)\big]$ such that 
        \[ \forall_{x \in \mathbf{C}(s)} \forall_{y \in \mathbf{D}(\gamma(s))} \big(x \in \dom(f) \wedge y \Vdash_s^\gamma x \Rightarrow y \in \dom(f') \wedge f'(y) \Vdash_t^\gamma f(x)\big). \]
     
    \end{enumerate}
 In this case we say that $f'$ tracks $f$, and we write $f' \Vdash_{(s,t)}^\gamma f$. We also write $\pmb{\gamma} \colon \mathbf{C \simto D}$ for a simulation $\pmb{\gamma}$ from $\mathbf{C}$ to $\mathbf{D}$. We call a simulation $\pmb{\gamma} \colon \mathbf{C} \simto \SetsB$ a $($covariant$)$ presheaf-simulation. The identity simulation $\B 1_{\B C} \colon \B C \simto \B C$ is the pair 
$\big(\id_T, (\Vdash^{\B \iota_{\B C}}_{t})_{t \in T}\big)$, where $x{'} \Vdash^{\B \iota_{\B C}}_{t} x
:\Leftrightarrow x{'} = x$, for every $x{'}, x \in \B C(t)$. If $\pmb\delta \colon \B D \simto \B E$, the composite 
simulation $\pmb\delta \circ \pmb\gamma \colon \B C \simto \B E$ is the pair $\big(\delta \circ \gamma, 
(\Vdash^{\B \delta \circ \B \gamma}_t)_{t \in T}\big)$, where the relation $\Vdash^{\B \delta 
\circ \B \gamma}_{t} \subseteq \B E\big(\delta(\gamma(t))\big) \times \B C(t)$ is defined by
$$z \Vdash^{\B \delta \circ \B \gamma}_{t} x :\Leftrightarrow \exists_{y \in \B D(\gamma(t))}\big(z \Vdash^{\B \delta}_{\gamma(t)} y \ \wedge \ y \Vdash^{\B \gamma}_{t} x\big).$$
\end{definition}

%Next we show that t
The following presheaf-simulations on a computability model $\B C$ correspond to the representable functors $\Hom(a, -)$ over $a$ in a category $\C C$.

\begin{example}\label{ex: presheafsim}
Let $\B C$ be a \textit{locally-small} computability model over $T$, i.e., the class $\B C[s,t]$ of computable functions from $\B C(s)$ to $\B C(t)$ is a set, for every $s, t \in T$. If $t_0 \in T$, the representable-simulation $\pmb{\gamma}_{t_0} \colon \mathbf{C} \simto \SetsB$ consists of the class-function $\gamma_{t_0} \colon T \to \Sets$, defined by $\gamma_{t_0}(t) := \B C[t_0, t]$, for every $t \in T$, and the forcing relations $\Vdash_{t}^{\gamma_{t_0}} \subseteq \B C[t_0, t] \times \B C(t)$, defined by
\[f \Vdash_{t}^{\gamma_{t_0}} x : \Leftrightarrow \exists_{y \in \dom(f)}\big(f(y) = x\big).\]
To show that $\pmb{\gamma}_{t_0}$ is a simulation, we also need to suppose that $\B C$ is \textit{left-regular} i.e., $\forall_{t \in T}\forall_{x \in \B C(t)}\exists_{f \in \B C[t_0, t]}\exists_{y \in \dom(f)}\big(f(y) = x\big)$. All computability models that include the constant functions are left-regular (such as Kleene's first model $K_1$ over $T = \{0\}$ with $\B C(0) = \D N$, and $\B C[0, 0]$ the Turing-computable partial functions from $\D N$ to $\D N$). If $f \in \B C[s,t]$, it is easy to show that $f^* \Vdash_{(s,t)}^{\gamma_{t_0}} f$, where $f^*$ is the total function from $\B C[t_0, s]$ to $\B C[t_0, t]$, defined by 
%the rule 
$f^*(g) := f \circ g$, for every $g \in \B C[t_0, s]$. A \textit{right-regularity} condition on a locally-small computability model is %similarly 
needed, to define the %corresponding 
contravariant representable-simulations $\B {\delta}_{t_0} \colon \B C \simto \SetsB$, where ${\delta}_{t_0} \colon \B C \to \Sets$ is defined by $\delta_{t_0}(t) := \B C[t, t_0]$, for every $t \in T$.
\end{example}

\section{The Grothendieck computability model}
\label{sec: groth}

The Grothendieck computability model is the categorical counterpart to the category of elements, a special case of the general categorical Grothendieck construction. A category $\C C$ is replaced by a computability model $\B C$, and a (covariant) presheaf $S \colon \C C \to \Sets$ by a (covariant) simulation $\pmb{\gamma}\colon \mathbf{C} \simto \SetsB$. Moreover, the first-projection functor is replaced by the first-projection-simulation.

\begin{proposition}
\label{prp: groth1}
%[Grothendieck model]
Let $\mathbf{C}$ be a computability model over the class $T$ together with a simulation $\pmb{\gamma}\colon \mathbf{C} \simto \SetsB$.
%We define the \emph{Grothendieck model} 
The structure $\Grothendieck{\mathbf{C}}{\pmb{\gamma}}$ with type names the class
%\[ \Grothendieck{\mathbf{C}}{\pmb{\gamma}} = \Big( \big( \Sigma(u,r)\big)_{(u,r) \in \mathcal{U}}, \big(\Sigma[(u,r),(v,s)]\big)_{(u,r),(v,s) \in \mathcal{U}}\Big) \]
% by the following data:
%\begin{itemize}[leftmargin = 1.5em]
	%\item The underlying class is the class
	\[ \Grothendieck{t \in T}{\pmb{\gamma}(t)} := \big\{(t,b) ~\vert~ t \in T \hbox{ and }b \in \gamma(t) \big\}, \] 
 with data types, for every $(t,b) \in \Grothendieck{t \in T}{\pmb{\gamma}(t)}$, the sets
%	\item For each $(t,r) \in \Sigma_{\pmb{\gamma}} T$ we set 
	\[ \Big(\Grothendieck{\mathbf{C}}{\pmb{\gamma}}\Big)(t,b) := \big\{y \in \mathbf{C}(t) ~\vert~ b \Vdash_t^\gamma y\big\}, \]
 and computable functions from $\Big(\Grothendieck{\mathbf{C}}{\pmb{\gamma}}\Big)(s,a)$ to $\Big(\Grothendieck{\mathbf{C}}{\pmb{\gamma}}\Big)(t,b)$ the classes
%	\[ \Big(\Grothendieck{\mathbf{C}}{\pmb{\gamma}}\Big)\big[(t,r),(v,s)\big] := 
 \[\Big\{ f \in \mathbf{C}[s,t] ~\vert~ \forall_{x \in \dom(f)}\Big( x \in  \Big(\Grothendieck{\mathbf{C}}{\pmb{\gamma}}\Big)(s,a) \Rightarrow 
 %w \Vdash_t^\gamma 
 f(x) \in \Big(\Grothendieck{\mathbf{C}}{\pmb{\gamma}}\Big)(t,b)\Big)\Big\}, \]
%\end{itemize}
is a computability model. The class-function $\pr_1 \colon \Grothendieck{t \in T}{\pmb{\gamma}(t)} \to T$, defined by the rule $(t,b) \mapsto t$, and the 
forcing relations, defined, for every $(t,b) \in \Grothendieck{t \in T}{\pmb{\gamma}(t)}$, by 
 \[ y{'} \Vdash_{(t,b)}^{\pr_1} y :\Leftrightarrow y{'} = y, \]
%for every $(t,b) \in \Grothendieck{t \in T}{\pmb{\gamma}(t)}$, 
determine the first-projection-simulation $\prbold_1 \colon \Grothendieck{\mathbf{C}}{\pmb{\gamma}} \simto \mathbf{C}$.  
       
\end{proposition}

\begin{proof}
We show that the computable functions include the identities and are closed under composition. Notice that the defining property of the computable functions in the Grothendieck model is equivalent to the condition 
$a \Vdash_s^\gamma x \Rightarrow b \Vdash_t^\gamma f(x)$, for every $x \in \dom(f)$. If $(t,b) \in \Grothendieck{t \in T}{\pmb{\gamma}(t)}$, then the identity on $\Grothendieck{\mathbf{C}}{\pmb{\gamma}}(t,b)$ is the identity on $\mathbf{C}(t)$, i.e., $\mathbf{1}_{\mathbf{C}(t)}$ is a computable function from $\Big(\Grothendieck{\mathbf{C}}{\pmb{\gamma}}\Big)(t,b)$ to itself: if $x \in \mathbf{C}(t)$, then the implication 
$b \Vdash_t^\gamma x \Rightarrow b \Vdash_t^\gamma x$ holds trivially.
If $g$ is a computable function from $\Grothendieck{\mathbf{C}}{\pmb{\gamma}}(t,b)$ to $\Grothendieck{\mathbf{C}}{\pmb{\gamma}}(u,c)$ and $f$ is a computable function from $\Grothendieck{\mathbf{C}}{\pmb{\gamma}}(s,a)$ to 
$\Grothendieck{\mathbf{C}}{\pmb{\gamma}}(t,b)$, then $g \circ f$ is a computable function from $\Grothendieck{\mathbf{C}}{\pmb{\gamma}}(s,a)$ to 
$\Grothendieck{\mathbf{C}}{\pmb{\gamma}}(u,c)$.
%This boils down to showing that given
For that, let 
%$x \in \Grothendieck{\mathbf{C}}{\pmb{\gamma}}$ 
$x \in \dom(f)$ and $f(x) \in \dom(g)$. If $a \Vdash_s^\gamma x$, then 
$b \Vdash_t^\gamma f(x)$, and hence $c \Vdash_u^\gamma g(f(x))$.
%$g\big(f(x)\big) \Vdash_r^{\pmb{\gamma}} w$. 
%    But this is immediate from the definition of $g$ as $f(x) \in \dom(g) %\cap \Grothendieck{\mathbf{C}}{\pmb{\gamma}}$.
Next, we show that $\prbold_1$ is a simulation.
%The first condition is trivially fulfilled, as for every
If $y \in \Grothendieck{\mathbf{C}}{\pmb{\gamma}}(t,b)$, then $x \Vdash_{(t,b)}^{\pr_1} x$, and if  
 %   For the second condition 
    $f$ is a computable function from $\Grothendieck{\mathbf{C}}{\pmb{\gamma}}(s,a)$ to $\Grothendieck{\mathbf{C}}{\pmb{\gamma}}(t,b)$,
then $f \Vdash_{((s,a),(t,b))}^{\pr_1}f$.
\end{proof}

%\begin{proposition}
 %   Let $\mathbf{C}$ be a computability model over $T$ and $\pmb{\gamma} \colon \mathbf{C} \simto \Sets$ be a simulation. 
 %   Then we obtain a simulation $\prbold_1 \colon \Grothendieck{\mathbf{C}}{\pmb{\gamma}} \simto \mathbf{C}$ defined through the following data: 
 %   \begin{itemize}
        %\item The underlying class function $\pr_1 \colon T_{\pmb{\gamma}} \to T$ is given through the rule $(t,u) \mapsto t$.
    %    \item The tracking relations are given as follows: For every $(t,r) \in T_{\pmb{\gamma}}$ we define 
  %      \[ y \Vdash_{(t,r)}^{\pr_1} x \Leftrightarrow y = x. \]
  %  \end{itemize}
%\end{proposition}

The following fact is straightforward to prove.

\begin{proposition}
Let $\mathcal{C}$ be a category and $S \colon \mathcal{C} \to \Sets$ a pullback-preserving presheaf on $\C C$. Let $\gamma^S \colon \mathcal{C}_0 \simto \SetsB$ be defined via $\gamma^S(c) = S(c)$ and the relations $\Vdash_c^{\gamma^S}$ are simply the diagonal, and let 
$\{\pr_2\} \colon \Grothendieck{\mathcal{C}}{S} \to \Sets $ be defined by $\{\pr_2\}(c,x) := \{x\}$ and if $f \colon (c,x) \to (d,y)$ in 
$\Grothendieck{\mathcal{C}}{S}$, let $[S(f)](x) := y$. Then
%we have that
    \[ \Grothendieck{\CM^{\prt}(\mathcal{C};S)}{\pmb{\gamma}^S} = \CM^{\prt}\Big( \Grothendieck{\mathcal{C}}{S};\{\pr_2\}\Big).
    \]
\end{proposition}

\begin{remark}
\label{rem: asm}
 The functor $\B C \mapsto \Assemblies(\B C)$ studied in~\cite{LN15} does not ``preserve'' the Grothendieck construction. Namely, if $\B 1$ is a
terminal computability model with type names $\{\emptyset\}$, data type $\B 1(\emptyset) = \{\emptyset\}$, and as only computable function the identity, then one can define a presheaf $\id_{\B 1} \colon \B 1 \simto \SetsB$, and show that  
% The above shows that we have a functor $\Grothendieck{(-)}{(-)}$ that sends a presheaf-simulation $\pmb{\gamma} \colon \C C \simto \Sets$ to its Grothendieck model. 
%    Longley showed that there also exists a functor that sends each computability model to its category of assemblies. 
%    One can then show that these two functors do not interact the same way as the Grothendieck construction did with the canonical computability model, in particular one obtains
    \[ \Assemblies\bigg(\Grothendieck{\B 1}{\id_{\B 1}}\bigg) \neq \Grothendieck{\Assemblies{(\B 1)}}{\Assemblies(\id_{\B 1})}. \]
\end{remark}

Next we show that the category of computability models $\CompMod$ is a type-category, in the sense of Pitts~\cite{Pi00}, pp.~110-111, a reformulation of Cartmell's categories with attributes in~\cite{Ca86}.
These categories are what we call $(\fam, \Sigma)$-categories with a terminal object in~\cite{Pe23}, and serve as categorical models of dependent type systems. First, we lift a simulation $\pmb{\gamma} \colon \mathbf{C \simto D}$ to a simulation between the Grothendieck computability models $\Grothendieck{\mathbf{C}}{(\pmb{\delta} \circ \pmb{\gamma})}$ and $\Grothendieck{\mathbf{D}}{\pmb{\delta}}$. 
%This is done in the following way:

\begin{lemma}\label{lem: lift}
    Let $\mathbf{C,D}$ be computability models over the classes $T,U$ respectively, and $\pmb{\gamma} \colon \mathbf{C \simto D}, \pmb{\delta} \colon \mathbf{D} \simto \SetsB$ simulations. 
    There is a simulation $\Grothendieck{\pmb{\delta}}{\pmb{\gamma}} \colon \Grothendieck{\mathbf{C}}{(\pmb{\delta} \circ \pmb{\gamma})} \simto \Grothendieck{\mathbf{D}}{\pmb{\delta}}$, such that the following is a pullback square 
    \[
    \begin{tikzcd}
        \Grothendieck{\mathbf{C}}{(\pmb{\delta} \circ \pmb{\gamma})} \ar[r,"\Grothendieck{\pmb{\delta}}{\pmb{\gamma}}",-simto] \ar[d,"\prbold_1",-simto] & \Grothendieck{\mathbf{D}}{\pmb{\delta}} \ar[d,"\prbold_1",-simto]
        \\
        \mathbf{C} \ar[r,"\pmb{\gamma}",-simto] & \mathbf{D}.
    \end{tikzcd}
    \]
 %   commute and the above is a pullback square
\end{lemma}

\begin{proof}
    To define 
   % the simulation 
    $\Grothendieck{\pmb{\delta}}{\pmb{\gamma}}$, let the 
underlying class-function $\Grothendieck{{\delta}}{{\gamma}} \colon \Grothendieck{t \in T}{\pmb{\gamma}(t)} \to \Grothendieck{u \in U}{\pmb{\delta}(u)}$ be defined by the rule $(t,b) \mapsto \big(\gamma(t), b)$. The corresponding forcing relations are defined by $x' \Vdash_{(t,b)}^{\Grothendieck{{\delta}}{{\gamma}}} x :\Leftrightarrow x' \Vdash_t^\gamma x.$ 
It is straightforward to show that $\Grothendieck{\pmb{\delta}}{\pmb{\gamma}}$ is a simulation. Next we show that the above square commutes. 
 On the underlying classes this is immediate as 
    \[ \pr_1\big(\Grothendieck{{\delta}}{{\gamma}}(t,b)\big) = \pr_1\big(\gamma(t)\big) = \gamma\big(\pr_1(t,b)\big).\]
    On the forcing relations we observe that if $x' \Vdash_{(t,b)}^{\pr_1 \circ \Grothendieck{{\delta}}{{\gamma}}} x$, then $x' \Vdash_{(t,b)}^{\Grothendieck{{\delta}}{{\gamma}}} x$, and thus $x' \Vdash_t^{\gamma} x$, which is also equivalent to $x'\Vdash_{(t,b)}^{\gamma \circ \pr_1} x$.
   Finally, we show the pullback property. Let a computability model $\mathbf{E}$ over a class $V$ with simulations $\pmb{\alpha},\pmb{\beta}$ be given, such that the following rectangle commutes
    \begin{equation} \begin{tikzcd}
        \mathbf{E} \ar[r,"\pmb{\beta}",-simto] \ar[d,"\pmb{\alpha}",-simto] & \Grothendieck{\mathbf{D}}{\pmb{\delta}} \ar[d,"\prbold_1",-simto]
        \\
        \mathbf{C} \ar[r,"\pmb{\gamma}",-simto] & \mathbf{D}.
        \end{tikzcd}  \label{PullbackGroth::eq::1}\end{equation}
 %   commutes.
    We find a unique simulation $\pmb{\zeta} \colon \mathbf{E} \simto \Grothendieck{\mathbf{C}}{(\pmb{\delta} \circ \pmb{\gamma})}$ such that both following triangles 
    \[\begin{tikzcd}
        \mathbf{E} \ar[dr,"\pmb{\zeta}",-simto] \ar[ddr,"\pmb{\alpha}",bend right = 30,-simto] \ar[drr,"\pmb{\beta}", bend left = 30,-simto] 
        \\
        &  \Grothendieck{\mathbf{C}}{(\pmb{\delta} \circ \pmb{\gamma})} \ar[r,"\Grothendieck{\pmb{\delta}}{\pmb{\gamma}}",-simto] \ar[d,"\prbold_1",-simto] & \Grothendieck{\mathbf{D}}{\pmb{\delta}} \ar[d,"\prbold_1",-simto]
        \\
        & \mathbf{C} \ar[r,"\pmb{\gamma}",-simto] & \mathbf{D}
    \end{tikzcd}\]
    commute. 
    First we define $\zeta$ on the level of the underlying classes. If $v \in V$, let $\zeta(v) = \big( \alpha(v), c\big)$, where 
    %where 
    $c \in \delta(\gamma(v))$ is the unique $c$ such that
%    be the unique $y$ 
    $\beta(v) = (u,c)$ for some $u$. Clearly, $\zeta$ is well-defined. Next we define the forcing relations. Let
    \[ x' \Vdash_v^{\zeta} x :\Leftrightarrow x' \Vdash_v^\alpha x. \]
    This relations are well-defined and in conjunction with the aforementioned class-function they constitute a simulation.
    Observe that the two triangles already commute on the level of the underlying class-functions, so it remains to check the forcing relations.
   % \begin{itemize}
   %     \item[``$\subseteq$''] 
        Assume we are given $v \in V$ and $x'' \in \mathbf{E}(v), x' \in \Big(\Grothendieck{\mathbf{D}}{\pmb{\delta}}\Big)\big(\beta(v)\big)$ and $x \in \mathbf{C}\big(\alpha(v)\big)$ such that 
        \[ x' \Vdash_v^{\beta} x'' \hbox{ and } x \Vdash_v^\alpha x''. \]
        By definition we have to show that there exist $y_1,y_2$ such that 
        \[ x' \Vdash_{\zeta(v)}^{\Grothendieck{{\delta}}{{\gamma}}} y_1 \hbox{ and } y_1 \Vdash_v^\zeta x'', \hbox{ and } x \Vdash_{\zeta(v)}^{\pr_1} y_2 \hbox{ and } y_2 \Vdash_v^{\zeta} x''.  \]
       We know that the square \eqref{PullbackGroth::eq::1} commutes and $x' \Vdash_{(\Grothendieck{{\delta}}{{\gamma}})(\zeta(v))}^{\pr_1} x'$, thus from $x'\Vdash_v^\beta x''$ we conclude that $x' \Vdash_v^{\gamma \circ \alpha} x''$. 
       This in turn ensures that there is $y$ such that $x' \Vdash_{\alpha(v)}^\gamma y$ and $y \Vdash_v^\alpha x''$. 
       By definition of $\Grothendieck{{\delta}}{{\gamma}}$ we then have that $x'\Vdash_{\alpha(v)}^{\Grothendieck{{\delta}}{{\gamma}}} y$ and thus $y$ is our desired $y_1$. 
       For $y_2$ we simply choose $x$ and it is easy to see that this fulfills the requirements. 
     %  \item[``$\supseteq$''] 
     The above implications also work in the reverse direction. 
  %  \end{itemize}
   % So we have seen that $\pmb{\zeta}$ makes the desired triangles commute. 
    It is immediate to show that $\pmb{\zeta}$ is the unique simulation making the triangles commutative, as it is determined by the definition of $\pmb{\beta}, \pmb{\alpha}$.
\end{proof}

\begin{lemma}\label{lem: strict}
If $\mathbf{C,D,E} \in \CompMod$, 
%are computability models. 
%together with simulations 
$\pmb{\gamma} \colon \mathbf{C \simto D}, \pmb{\delta} \colon \mathbf{D \simto E}$, 
%simulations, 
and $\pmb{\epsilon} \colon \mathbf{E} \simto \SetsB$ is a presheaf-simulation, then the following strictness conditions hold:\\[1mm]
\normalfont (i)
\itshape $\Grothendieck{\pmb{\epsilon}}{\B 1_{\B E}}  = \mathbf{1}_{\Grothendieck{\mathbf{E}}{\pmb{\epsilon}}}$.\\[1mm]
\normalfont (ii)
\itshape    $\Grothendieck{\pmb{\epsilon}}{(\pmb{\delta} \circ \pmb{\gamma})}
        =
        \Grothendieck{\pmb{\epsilon}}{\pmb{\delta}} \circ 
        \Grothendieck{(\pmb{\epsilon}\circ \pmb{\delta})}{\pmb{\gamma}}$.
\end{lemma}

\begin{proof}
(i) It suffices to observe that by its definition the simulation $\Grothendieck{\pmb{\epsilon}}{\B 1_{\B E}}$ on the level of the underlying class takes a pair $(t,u)$ to $(\mathbf{1}_{\mathbf{E}}(t),u) = (t,u)$, so on the level of the underlying class-functions the two simulations agree.
%$\Sigma(\mathbf{1_E})$ and $\mathbf{1}_{\Grothendieck{\mathbf{E}}{\pmb{\epsilon}}}$ agree.
For the forcing relations we see that both simulations are the corresponding diagonal.\\
%$\Vdash_t^{\Sigma(1_E)}$ is simply the diagonal, the same as $\Vdash_{(t,u)}^{1_{\Sigma_\delta \mathbf{E}}}$. \\
%Hence the two simulations agree.\\
(ii) 
To verify this equation on the level of underlying classes we
have that
%simply compute that
\[ \Grothendieck{\pmb\epsilon}{(\pmb\delta \circ \pmb\gamma)}(t,b) = \big(t, (\delta \circ \gamma)(b)\big)= \Grothendieck{\pmb\epsilon}{\pmb\delta}\big( t,\gamma(b)\big) = \Grothendieck{\pmb\epsilon}{\pmb\delta}\Big(\Big(\Grothendieck{\pmb{\epsilon}\circ\pmb\delta}{\pmb\gamma}\Big)(t,b)\Big).\]
For the forcing relations we simply remark that $x \Vdash_{(t,b)}^{\sum_{\epsilon}{\delta \circ \gamma}} y$ if and only if $x \Vdash_t^{\delta \circ \gamma} y$.
%by definition, 
Similarly, we have that 
$x \Vdash_{(t,b)}^{\sum_{\epsilon}{\delta}} y$ if and only if 
%\hbox{ iff } 
$x \Vdash_t^{\delta} y$, and
%\hbox{ and }
$x \Vdash_{(t,b)}^{\sum_{\epsilon \circ \delta }{ \gamma}} y$
%\hbox{ iff }
if and only if $x \Vdash_t^{\gamma} y$.
%\]
Hence,
%we now have that 
$x \Vdash_{(t,b)}^{\sum_{\epsilon}{\delta} \circ \sum_{\epsilon \circ \delta}\gamma} y$ if and only if 
$x \Vdash_t^{\delta \circ \gamma}z$, which by the above is equivalent to $x \Vdash_{(t,b)}^{\sum_{\epsilon}{\delta \circ \gamma}} z$.
%by the pasting lemma for pullbacks the outer rectangle is also a pullback. As $\Grothendieck{\mathbf{C}}{\pmb{\epsilon}\circ (\pmb{\delta}\circ \pmb{\gamma}} = \Grothendieck{\mathbf{C}}{(\pmb{\epsilon}\circ \pmb{\delta})\circ \pmb{\gamma}}$, by considering the the simulation $\Grothendieck{\pmb{\epsilon}}{(\pmb{\delta} \circ \pmb{\gamma}} \colon \Grothendieck{\mathbf{C}}{\pmb{\epsilon}\circ (\pmb{\delta}\circ \pmb{\gamma})} \simto \Grothendieck{\mathbf{E}}{\pmb{\epsilon}}$ and the identity simulation between $\Grothendieck{\mathbf{C}}{\pmb{\epsilon}\circ (\pmb{\delta}\circ \pmb{\gamma}}$ and $\Grothendieck{\mathbf{C}}{(\pmb{\epsilon}\circ \pmb{\delta})\circ \pmb{\gamma}}$, we get the required equality by the pullback property.
\end{proof}

\begin{theorem}\label{thm: typecat}
The category $\CompMod$ is a type-category.
\end{theorem}

\begin{proof}
This follows immediately from Lemma~\ref{lem: lift}, Lemma~\ref{lem: strict}, and the fact that $\CompMod$ has a terminal object, as explained in Remark~\ref{rem: asm}.
\end{proof}

\section{Fibration-simulations and opfibration-simulations}
\label{sec: fibr}

The (covariant) Grothendieck construction allows the generation of fibrations  (opfibrations), as the first-projection functor
$\pr_1 \colon \Grothendieck{\C C}{P} \to \C C$ is a (split) opfibration, if $P$ is a covariant presheaf, or a (split) fibration, if $P$ is a contravariant presheaf. In this section we introduce the notion of a fibration and opfibration-simulation and we show that the first-projection-simulation $\B {\pr_1} \colon \Grothendieck{\mathbf{C}}{\pmb{\gamma}} \simto \B C$ is a (split) opfibration-simulation, as we work with covariant presheaf-simulations. The dual result is shown similarly.

\textit{In this section, $\B E$ is a computability model over $T$ and $\B B$ a computability model over $U$. Moreover, the pair 
$$\pmb{\varpi} := \bigg(\varpi \colon \colon T \to U, \ \big(\Vdash^{\varpi}_t\big)_{t \in T}\bigg)$$ 
is a simulation of type $\B E \simto \B B$.}
In contrast to what it holds for functors, for simulations $\pmb{\gamma} \colon \mathbf{E \rightarrowtriangle B}$ each computable function $f$ in $\mathbf{E}$ is tracked, in general, by a multitude of maps $f'$ in $\mathbf{B}$.
Thus, for each opspan 
\[ 
\begin{tikzcd}
\mathbf{E}(t_1) & \mathbf{E}(t_2) \ar[from = r,"f",-simto] \ar[from = l,"g"',-simto] & \mathbf{E}(t_3) 
\end{tikzcd} \]
we have a whole {class}, in general, of opspans 
\[
\begin{tikzcd}
\mathbf{B}(\gamma(t_1)) &\mathbf{B}(\gamma(t_2)) \ar[from = r,"f'",-simto] \ar[from = l,"g'"',-simto] & \mathbf{B}(\gamma(t_3))
\end{tikzcd} 
\]
such that $f'$ tracks $f$ and $g'$ tracks $g$.
%\\\textbf{Notation:}
%If $f\colon A \to B, g \colon B \to C$ are partial functions of sets we set 
%\[ \dom(g \circ f) := \{ x \in \dom(f) ~\vert~ f(x) \in \dom(g) %\}. \]

\begin{definition}[Cartesian computable function]
%Let $\pmb{\varpi} \colon \mathbf{E \rightarrowtriangle B}$ be a simulation. 
Let $f' \in \B B[s,s']$ and $t' \in T$, such that $\varpi(t') = s'$ be given. 
We call a computable function $f \in \mathbf{E}[t,t']$  {cartesian} for $f'$ and $t'$, if 
%the following conditions holds:
%\begin{enumerate}
 %   \item 
    $f' \Vdash_{(t,t')}^{\varpi} f$, and
%\[ \forall_t'' \forall_{g \in \mathbf{C}[t'',t']} \forall_{f'} \forall_{g'} \forall_{h \in \mathbf{B}[\gamma(t''), \gamma(t)]}
%\exists_{h \in \mathbf{C}(t'', t)} ( f' \Vdash_{(t, t')}^\gamma f \wedge g' \Vdash_{(t'',t')} g  \wedge f \circ h = g  \wedge k \Vdash_{(t'',t')}^\gamma h). \]
%\item 
given computable functions $g \in \mathbf{E}[t'',t'], g' \in \B B[\varpi(t''),\varpi(t')]$, and $h \in \mathbf{B}[\varpi(t''), \varpi(t)]$ %such 
as in the following diagram 
\[ \begin{tikzcd}
&\mathbf{E}(t)   \ar[rr,"\Vdash_t^\varpi",dashed, no head]&& \mathbf{B}(\varpi(t)) \ar[dd,"f'",harpoon]
\\
\mathbf{E}(t'') \ar[ur,"k",harpoon] \ar[rr,"\Vdash_{t''}^\varpi",dashed, no head, near start] \ar[dr,"g"] && \mathbf{B}(\varpi(t'')) \ar[dr,"g'",harpoon]\ar[ur,"h",harpoon]
\\
&\mathbf{E}(t') \ar[rr,"\Vdash_{t'}^\varpi",dashed, no head]&& \mathbf{B}(\varpi(t')) 
\ar[from = 1-2, to = 3-2,"f", crossing over, near end,harpoon]
\end{tikzcd} \]
that is $g'$ tracks $g$, there is some $k \in \mathbf{E}[t'',t]$ satisfying the following property:
$h \Vdash_{(t'',t)}^\varpi k$, and for every $x \in \B E(t''), y \in \mathbf{B}(\varpi(t''))$, such that
%\[ 
$y \Vdash_{t''}^\varpi x$,
%\quad 
$y \in \dom(f' \circ h) \cap \dom(g'),$ and $f'(h(y)) = g'(y)$,
%\]
then $x \in \dom(f \circ k) \cap \dom(g)$ and $g(x) = f(k(x))$.
%\end{enumerate}
\end{definition}

\begin{definition}[Opcartesian computable function]
Let $f' \in \B B[s',s]$ 
%be a computable function 
and $t' \in T$, such that $\varpi(t') = s'$ be given
%Let $\pmb{\varpi} \colon \mathbf{E \rightarrowtriangle B}$ be a simulation. 
We call a computable function $f \in \mathbf{E}[t',t]$  {opcartesian} for $f'$ and $t'$, if 
%the following holds:
%\[ \forall_t'' \forall_{g \in \mathbf{C}[t'',t']} \forall_{f'} \forall_{g'} \forall_{h \in \mathbf{B}[\gamma(t''), \gamma(t)]}
%\exists_{h \in \mathbf{C}(t'', t)} ( f' \Vdash_{(t, t')}^\gamma f \wedge g' \Vdash_{(t'',t')} g  \wedge f \circ h = g  \wedge k \Vdash_{(t'',t')}^\gamma h). \]
%\begin{enumerate}
%\item 
$f' \Vdash_{(t',t')}^\varpi f$, and 
%\item 
given computable functions $g \in \mathbf{E}[t',t''],g' \in \B B[\varpi(t'),\varpi(t'')]$ and $h \in \mathbf{B}[\varpi(t),\varpi(t'')]$ as in the following diagram
\[ \begin{tikzcd}
&\mathbf{E}(t)  \ar[dl,"l",harpoon] \ar[from = dd,"f", near end,harpoon]  \ar[rr,"\Vdash_t^\varpi",dashed, no head]&& \mathbf{B}(\varpi(t)) \ar[from = dd,"f'",harpoon]
\\
\mathbf{E}(t'')  \ar[rr,"\Vdash_{t''}^\varpi",dashed, no head, near end,crossing over] \ar[from = dr,"g"] && \mathbf{B}(\varpi(t'')) \ar[from= dr,"g'",harpoon]\ar[from = ur,"h",harpoon]
\\
&\mathbf{E}(t') \ar[rr,"\Vdash_{t'}^\varpi",dashed, no head]&& \mathbf{B}(\varpi(t')) 
\end{tikzcd} \]
that is $g'$ tracks $g$, there is some  $l \in \mathbf{E}[t,t'']$ satisfying the following property:
$h$ tracks $l$, and
for every $x \in \mathbf{E}(t'),y \in \mathbf{B}(\varpi(t'))$, such that
%\[ 
$y \Vdash_{t'}^\varpi x$, $y \in \dom(h \circ f') \cap \dom(g')$,
and $f'(h(y)) = g'(y)$, 
%\]
then $x \in \dom(l \circ f) \cap \dom(g)$ and $g(x) = l(f(x))$.
%\end{enumerate}
\end{definition}

Note that the computable functions $k \in \mathbf{E}[t'',t]$ and $l \in \mathbf{E}[t, t'']$ in the above two definitions, respectively, are not unique.
%the existence of the  another important distinction to categories: the existence we require is not unique.

\begin{definition}[Fibration-simulation]
%Let $\mathbf{B,E}$ be computability models over $T$ and $U$ respectively. 
We call $\pmb{\varpi} \colon \mathbf{E \rightarrowtriangle B}$ a  fibration-simula-\\
tion, if for every computable function $f \in \mathbf{B}\big[u, \varpi(t)\big]$ there is $g \in \mathbf{E}[t',t]$ cartesian for $f$ and $t$. In this case, we call $g$ a lift of $f$.
\end{definition}

\begin{definition}[Opfibration-simulation]
  %  Let $\mathbf{B,E}$ be computability models over $T,U$ respectively. 
    We call $\pmb{\varpi} \colon \mathbf{E \to B}$ an opfibration-simulation, if for every computable function $f \in \mathbf{B}\big[ \varpi(t),u\big]$, there is $g \in \mathbf{E}[t,t']$ opcartesian for $f$ and $t$. In this case, we call $g$ a lift of $f$.
\end{definition}

\begin{example}
    Let $\mathcal{E,B}$ be categories with presheaves $S,S'$ and let $F \colon \mathcal{E \to B}$ be a fibration, such that $S' \circ F = S$. 
    Then, 
    %the simulation 
    $\pmb{\gamma}^F \colon \CM^{\total}(\mathcal{E};S) \simto \CM^{\total}(\mathcal{B};S')$
    is a fibration-simulation. 
    To see this, assume we are given a computable function in $\CM^{\total}(\mathcal{B};S')$, that is a function $S'(f) \colon S'(b) \to S'(b')$, and $e \in \mathcal{E}$ such that $F(e') = b'$.
    As $F$ is a fibration, we find an arrow $g \colon e \to e'$ cartesian over $f$ and $b'$ . 
    We show that $S(g)$ is the desired cartesian function over $S'(f)$ and $S(b')$. 
    For this, let functions $S(h),S(h_2),S(g_2)$ as in the following diagram,
    %situation as in 
    \[ \begin{tikzcd}
        & S(e) \ar[dd,"S(g)" near end] \ar[rr,"\Vdash_e^{\gamma^F}", dashed, no head] && S'(b) \ar[dd,"S'(f)"] \ar[from = dl,"S'(h)"] 
        \\
        S(e'')\ar[ur,"S(k)"]  \ar[dr,"S(g_2)"]\ar[rr,"\Vdash_{e''}^{\gamma^F}"near end , dashed, no head, crossing over]&& S'(b'') \ar[dr,"S'(h_2)"]
        \\
        & S(e') \ar[rr,"\Vdash_{e'}^{\gamma^F}", dashed, no head] &&S'(b')
    \end{tikzcd} \]
    be given, where we used that $\CM^{\prt}(\mathcal{E};S)(e) = S(e)$ and $\CM^{\prt}(\mathcal{B};S')(b) = S(b)$, for every $e$ and $b$, respectively.  
    As $g$ is cartesian over $f$ and $b'$,  we obtain an arrow $k \colon e'' \to e$, such that $g \circ k = g_2$ and $F(k) = h_2$.
    Obviously, $S(k)$ is the function needed, and hence $S(g)$ is cartesian over $S'(f)$ and $S(b')$. 
\end{example}

\begin{proposition}
\label{prp: projopfibr}
If $\mathbf{C}$ is a computability model and $\gamma \colon \mathbf{C} \to \SetsB$ a simulation, then the first-projection-simulation $\prbold_1 \colon \Grothendieck{\mathbf{C}}{\pmb{\gamma}} \simto \mathbf{C} $ is an opfibration-simulation.
\end{proposition}

\begin{proof}
Assume we are given a computable function $f \in \mathbf{C}[t,t']$ and $\pr_1(t,b) = t$.
We need to find some $b \in \mathbf{C}(t')$, such that $\pr_1(t',b') = t'$, and a computable function $f' \in \Big(\Grothendieck{\mathbf{C}}{\pmb\gamma}\Big)\big[(t,b),(t',b')\big]$, such that $f \Vdash_{((t,b),(t',b'))}^{\pr_1} f'$.
By definition we know that $f \Vdash_{((t,b),(t',b'))}^{\pr_1} f'$ if and only if $f = f'$, so we have to find $y \in \mathbf{C}(t')$, such that $f(b) = b'$. For this, we simply take $b' := f(b)$. 
To show that $f$ is opcartesian for $f$ and $b$, we consider the following diagram  
\[ \begin{tikzcd}
    & \Big(\Grothendieck{\B C}{\pmb\gamma}\Big)(t,b) \ar[dd,"f" near end, harpoon] \ar[dl,"g",harpoon] \ar[rr,"\Vdash_{(t,b)}^{\pr_1}",dashed, no head] && \mathbf{C}(t)
    \ar[dd,"f", harpoon] \ar[dl,"g",harpoon] 
    \\
    \Big(\Grothendieck{\B C}{\pmb\gamma}\Big)(t'',b'') \ar[dr,"h",harpoon] \ar[rr,"\Vdash_{(t'',b'')}^{\pr_1}",dashed, no head, near end, crossing over]&& \mathbf{C}(t'')\ar[dr,"h",harpoon]
    \\
     & \Big(\Grothendieck{\B C}{\pmb\gamma}\Big)  (t',b')\ar[rr,"\Vdash_{(t',b')}^{\pr_1}",dashed, no head] && \mathbf{C}(t')
\end{tikzcd}\]
and we observe that $h$ fills also the triangle on the left, as we have that $f = h \circ g$ whenever they are defined, so in particular $f(b) = h\big(g(b)\big)$, and thus $b' = h(b'')$. Hence, $h$ is a computable function from $\Big(\Grothendieck{\B C}{\pmb\gamma}\Big)(t'',b'')$ to $\Big(\Grothendieck{\B C}{\pmb\gamma}\Big)(t',b')$. 
\end{proof}

Next we define 
%the notions of a 
split fibration-simulations and split opfibration-simulations. 

\begin{definition}
    A  {splitting} for a fibration-simulation $\pmb{\varpi} \colon \mathbf{E \simto B}$
    %where $\mathbf{B,E}$ live over the classes $T,U$ respectively,
    is a rule $\varpi^\triangle$ that corresponds a pair $(f,u)$, where $f\in \mathbf{B}[t_1,t_2]$ and $\varpi(u) = t_2$, to a function $f' \in \mathbf{E}[u,u']$ cartesian for $f$ and $u$.
    This rule $\varpi^\triangle$ is subject to the following conditions:
    \begin{itemize}
        \item For every $f \in \mathbf{B}[t_1,t_2]$ and every $g \in \mathbf{B}[t_2,t_3]$ we have that 
        \[ \varpi^\triangle(g \circ f, u_1) = \varpi^\triangle(g,u_1) \circ \varpi^\triangle(f,u_2).\]
        \item For every $t \in T$ we have that 
        $ \varpi^\triangle(\mathbf{1}_{\mathbf{B}(t)},u) = (\mathbf{1}_{\mathbf{E}(u)},u).$
    \end{itemize}
    A  {splitting} for an opfibration-simulation $\pmb{\varpi} \colon \mathbf{E \simto B}$
    %, where $\mathbf{B,E}$ live over the classes $T,U$ respectively, 
    is a rule $\varpi^\triangle$ that corresponds a pair $(f,u)$, where $f\in \mathbf{B}[t_1,t_2]$ and $\varpi(u) = t_1$, to a function $f' \in \mathbf{E}[u,u']$ opcartesian over $f$ and $u$.
    This rule $\varpi^\triangle$ is subject to the following conditions:
    \begin{itemize}
        \item For every $f \in \mathbf{B}[t_1,t_2]$ and every $g \in \mathbf{B}[t_2,t_3]$ we have that 
        \[ \varpi^\triangle(g \circ f, u_1) = \varpi^\triangle(g,u_2) \circ \varpi^\triangle(f,u_1).\]
        \item For every $t \in T$ we have that 
        $\varpi^\triangle(\mathbf{1}_{\mathbf{B}(t)},u) = (\mathbf{1}_{\mathbf{E}(u)},u).$
    \end{itemize}
    A $($op$)$fibration-simulation $\pmb{\varpi}$ is split, if it admits a splitting $\varpi^\triangle$. 
\end{definition}

\begin{corollary}
\label{cor: split}
The simulation $\prbold_1 \colon \Grothendieck{\B C}{\pmb\gamma} \simto \mathbf{C}$ is a split opfibration-simulation.
\end{corollary}

\begin{proof}
We can simply take $\prbold_1^\triangle$ to be defined by the rule $\prbold_1^\triangle(f,u) := (f,u)$.
\end{proof}
%\subsection{A 2-categorical approach}

%Our approach to fibrations and cartesian arrows is possibly a bit different than what the 2-categorically inclined mind envisioned. 
%Indeed, in the works of Riehl \cite{Ri10} and Wong \cite{Wo19} one sees a 2-categorical approach to cartesian arrows and fibrations that should also be translatable to our framework, as the category $\CompMod$ of computability models constitutes a 2-category. 
%One can show that our approach to cartesian arrows can be translated to this language.

\section{Conclusions and future work}
\label{sec: concl}
\vspace{-1mm}
In~\cite{LN15} many concepts and results from category theory were translated to the theory of computability models, where equalities between arrows are replaced by certain relations between type names and (partial) computable functions. In this paper we extended the
work initiated in~\cite{Pe22a,Pe22b} by translating the Grothendieck construction and the notions of fibration and opfibration within computability models. The category $\CompMod$ was shown to be a type-category, a fact that allows the transport of concepts and facts from the theory of type-categories to the theory of computability models. For example, the simulations 
$$\pmb{\phi} \colon \B C \simto \Grothendieck{\mathbf{C}}{\pmb{\gamma}} \ \ ; \ \ \ \B {\pr_1} \circ \pmb{\phi} = \B 1_{\mathbf{C}}$$ are the canonical \textit{dependent functions} over $\B C$ and $\pmb{\gamma}$ (see also~\cite{Pe23}, Theorem 4.6). It is natural to search whether these dependent functions determine a computability model.
The following table includes the correspondences between categorical and computability model theory-notions presented here. \\[3mm]
%\vspace{3mm}
\centerline{
\begin{tabular}{l@{\hspace{.5cm}}l}
\itshape Category theory & \itshape Computability model theory \\
\hline
 category  $\C C$ &     computability model        $\B C$ \\
functor $F \colon \C C \to \C D$ & simulation $\gamma \colon \B C \simto \B D$\\
category of $\Sets$ &  computability model of $\SetsB$\\
presheaf $P \colon \C C \to \Sets$ & presheaf-simulation $\gamma \colon \B C \simto \SetsB$\\
representable functor $\Hom(a, -)$ & representable simulation $\B {\gamma}_{t_0}$\\
representable functor $\Hom(-, a)$& representable-simulation $\B {\delta}_{t_0}$\\
Grothendieck category $\Grothendieck{\C C}{P}$& Grothendieck computability model $\Grothendieck{\mathbf{C}}{\pmb{\gamma}}$\\
\parbox{0.4\textwidth}{first-projection functor \[\pr_1 \colon \Grothendieck{\C C}{P} \to \C C\]}& \parbox{0.4\textwidth}{first-projection-simulation \[\B {\pr_1} \colon \Grothendieck{\mathbf{C}}{\pmb{\gamma}} \simto \B C\] }\\
(op)cartesian arrow & (op)cartesian computable function\\
(op)fibration $\pi \colon \C E \to \C B$ & (op)fibration-simulation $\pmb{\varpi} \colon \B E \simto \B B$\\
split (op)fibration  & split (op)fibration-simulation\\\hline
\end{tabular}}

\vspace{3mm}
It is natural to ask whether the category of presheaves, or more generally of all functors between two categories, can be translated within computability models. As a consequence, a Yoneda-type embedding and a corresponding Yoneda lemma for computability models and appropriate presheaf-simulations can be formulated. In such a framework the Grothendieck computability model is expected to have the same crucial role to the proof of a corresponding density theorem with that of the Grothendieck category to the proof of the categorical density theorem. 
For that, we introduce forcing and tracking-moduli in the definition of a simulation i.e., realisers for the forcing and tracking relations. A \textit{forcing-modulus} $\phi^{\gamma}$ for $\pmb{\gamma} \colon \B C \simto \B D$ is a family of functions
%\vspace{2mm}
$$\phi^{\gamma} := \big(\phi^{\gamma}_t\big)_{t \in T} \ \ ; \ \ \ \ 
\phi^{\gamma}_t \colon \B C(t) \to \B D\big(\gamma(t)\big),$$
such that $\phi^{\gamma}_t(x) \Vdash^{\gamma}_t x$, for every $x \in \B C(t)$. 
A \textit{tracking-modulus} $\mu^{\gamma}$ for $\pmb {\gamma}$ is a family of functions
$$\mu^{\gamma} := \big(\mu^{\gamma}_{(s,t)}\big)_{s, t \in T} \ \ ; \ \ \ \ \mu^{\gamma}_{(s,t)} \colon \B C[s,t] \to \B D[\gamma(s), \gamma(t)],$$
such that $\mu^{\gamma}_{(s,t)}(f) \Vdash^{\gamma}_{(s,t)} f$, for every $f \in \B C[s,t]$. Conversely, if $\gamma \colon T \to U$, $\phi^{\gamma}$, and $\mu^{\gamma}$ are typed as above, such that for every $f \in \B [s, t]$ the following rectangle 
	\begin{center}
				\begin{tikzpicture}
					
					\node (E) at (0,0) {$\B D(\gamma(s))$};
					\node[right=of E] (F) {$\B D(\gamma(t))$};
					\node[above=of F] (A) {$\B C(t)$};
					\node [above=of E] (D) {$\B C(s)$};
					%\node [below=of D] (G) {$V$};

					\draw[-{Computer Modern Rightarrow[left]}] (E) to node [midway,below]{$\mu^{\gamma}_{(s,t)}(f)$}(F);
					\draw[-{Computer Modern Rightarrow[left]}] (D) to node [midway,above] {$f$}(A);
					\draw[->] (D) to node [midway,left] {$\phi^{\gamma}_s$}  (E);
					\draw[->] (A) to node [midway,right] {$\phi^{\gamma}_t$} (F);
					%\draw[->] (D)--(E) node [midway,left] {$f$};

				\end{tikzpicture}
			\end{center}
with total and partial functions commutes, in the obvious sense, then if we define $y \Vdash^{\gamma}_t x :\Leftrightarrow y = \phi^{\gamma}_t(x)$, we get a \textit{realised simulation} $\B C \simto \B D$. We can define then the \textit{exponential computability model} with type names the \textit{realised simulations}.
%, and with computable functions the similarly defined witnesses of the transformability relation between simulations.
We hope to elaborate on these concepts in the future.
%Due to lack of space though, we cannot elaborate on these concepts here. 
% partial homomorphisms in (swap) algebras
% If $\B {\pi} \colon \B E \simto \B B$ is a fibration-simulation, to lift a (c0) base of computability from $\B B$ to $\B E$.... Shall we include this here?

 Our approach to (op)fibration-simulations and (op)cartesian functions is different from the $2$-categorical approach to fibrations
 in~\cite{Ri10,Wo19}. In a subsequent work 
 %future extension of this work 
 we expect to show that our approach is equivalent to the
 %aforementioned 
 $2$-categorical one.
 % what a 2-categorically inclined mind would expect. Indeed, in the work of Riehl \cite{Ri10} and Wong \cite{Wo19} one sees a $2$-categorical approach to cartesian arrows and fibrations that is also translatable in our framework, as the category $\CompMod$ of computability models constitutes a $2$-category. 
 %to cartesian arrows can be translated to this language.

\end{document}